\documentclass[psamsfonts, 10pt]{amsart}

\usepackage{amssymb}
\usepackage{amsmath}
\usepackage{amsthm}
\usepackage[urlcolor=black, colorlinks=true, citecolor=black, linkcolor=black]{hyperref}
\usepackage[english]{babel}
\usepackage{graphics}
\usepackage{enumerate}
\usepackage[list=off]{caption}

\usepackage{mathtools}

\usepackage{tikz}
\tikzset{node distance=3cm, auto}
\usetikzlibrary{calc}
\usetikzlibrary{patterns}

\newcommand{ \R } { \mathbb{R} }
\newcommand{ \N } { \mathbb{N} }

\newcommand{\Q} {\mathbb{Q}}
\newcommand{\Hstar}{\mathbb{H}^*}
\newcommand{\Rstar}{\mathbb{R}^*}
\newcommand{\w}{\omega}

\newcommand{\cont}{\mathfrak{c}}

\newcommand{\script}{\mathcal}
\newcommand{\parentheses}[1]{{\left( {#1} \right)}}

\newcommand{\p}{\parentheses}
\newcommand{\of}{\parentheses}
\newcommand{\closure}[1]{\overline{#1}}
\newcommand{\closureIn}[2]{\closure{#1}^{#2}}
\newcommand{\interior}[1]{\mathrm{int}\of{#1}}
\newcommand{\interiorIn}[2]{\mathrm{int}_{#2}\of{#1}}

\newcommand{\Set}[1]{{\left\lbrace {#1} \right\rbrace}}
\newcommand{\singleton}{\Set}

\newcommand{\Union}{\bigcup}

\newcommand{\cardinality}[1]{{\left\lvert {#1} \right\rvert}}

\def\set#1:#2{\Set{{#1} \colon {#2}}}

\newcommand{\singletonDeletion}[1]{\script{D}\parentheses{#1}}
\newcommand{\localends}[1]{{\textnormal{ends}\p{#1}}}

\newcommand{\Remainder}[2]{\script{R}_{#1}\parentheses{#2} }

\begin{document}
\title[Reconstructing Topological Graphs and Continua]{Reconstructing Topological Graphs and Continua}
\author[Gartside, Pitz, Suabedissen]{Paul Gartside, Max F.\ Pitz, Rolf Suabedissen}
\thanks{This research formed part of the second author's thesis at the University of Oxford \cite{thesis}, supported by an EPSRC studentship.}
\address{The Dietrich School of Arts and Sciences\\301 Thackeray Hall\\Pittsburgh PA 15260}
\address{Mathematical Institute\\University of Oxford\\Oxford OX2 6GG\\United Kingdom}
\email{gartside@math.pitt.edu}
\email[Corresponding author]{max.f.pitz@gmail.com}
\email{rolf.suabedissen@maths.ox.ac.uk}

\begin{abstract} The deck of a topological space $X$ is the set $\singletonDeletion{X}=\set{[X \setminus \singleton{x}]}:{x \in X}$, where $[Z]$ denotes the homeomorphism class of $Z$. A space $X$ is \emph{topologically reconstructible}  if whenever $\singletonDeletion{X}=\singletonDeletion{Y}$ then  $X$ is homeomorphic to $Y$. 

It is shown that all metrizable compact connected spaces are reconstructible. It follows that all finite graphs, when viewed as a 1-dimensional cell-complex, are reconstructible in the topological sense, and more generally, that all compact graph-like spaces are reconstructible. 
\end{abstract}

\maketitle
\thispagestyle{plain}

\newtheorem*{recresult}{Theorem \ref{fincompactrec}}
\newtheorem{mythm}{Theorem} \numberwithin{mythm}{section} 
\newtheorem{myprop}[mythm]{Proposition}
\newtheorem{myobs}[mythm]{Observation}
\newtheorem{mycor}[mythm]{Corollary}
\newtheorem{mylem}[mythm]{Lemma} 
\newtheorem{myquest}[mythm]{Question} 
\newtheorem*{myconj}{Conjecture}
\newtheorem{mydef}[mythm]{Definition}
\newtheorem{myclaim}[mythm]{Claim}
\newtheorem{myremark}[mythm]{Remark}
\newtheorem{mycase}{Case}
\newtheorem{mycase2}{Case}
\newtheorem*{myclmn}{Claim}

\section{Introduction}

The well-known \emph{Graph Reconstruction Conjecture} asks if every non-trivial finite graph can be reconstructed from its deck of cards. More precisely, if $G$ is a finite graph and $x$ a vertex of $G$, then let $G-x$ be the graph obtained by deleting $x$ and all incident edges from $G$, and write $[G-x]$ for the isomorphism class of $G-x$. Then each $[G-x]$ is called a \emph{card} of $G$, and the set of all cards, $\singletonDeletion{G}=\{[G-x] : x \in G\}$, is the \emph{deck} of $G$. In 1941, Ulam and Kelly conjectured  that if $G$ is finite graph with at least three vertices and $H$ is another finite graph with the same deck, $\singletonDeletion{G}=\singletonDeletion{H}$, then $G$ and $H$ are isomorphic.  (For a survey of the reconstruction conjecture see for example \cite{Bondy}.)

In this paper we introduce and give a positive solution to an analogous topological reconstruction problem for finite topological graphs, and many other spaces including all metric continua (compact, metrizable connected spaces). For a space $X$ we denote by $[X]$ the homeomorphism class of $X$. The deck of $X$ is the set $\singletonDeletion{X}=\set{[X \setminus \singleton{x}]}:{x \in X}$, and cards of $X$ correspond to elements of the deck of $X$. The \emph{topological reconstruction problem} then asks for necessary and sufficient conditions on a space to be reconstructible from its deck---when does $\singletonDeletion{X}=\singletonDeletion{Y}$ imply $X$ homeomorphic to $Y$? (For notational convenience we often write `$Z$' instead of `$[Z]$' when talking about cards. In practice this causes no confusion.)

If we delete a point from the real line $\R$, we get a space homeomorphic to two copies of the real line. In other words, $\singletonDeletion{\R}=\Set{\R \oplus \R}$. Similarly $\singletonDeletion{\R^n}=\Set{\R^n \setminus \singleton{0}}$ for all $n \geq 1$. In the case of the unit interval $I$ we have $\singletonDeletion{I}=\Set{[0,1), [0,1) \oplus [0,1)}$. For the sphere, stereographic projection gives $\singletonDeletion{S^n}=\Set{\R^n}$. The second two authors have shown, \cite{recpaper}, that all these spaces are reconstructible, as are the rationals, which has deck $\singletonDeletion{\Q} = \Set{\Q}$, and the irrationals, $P$, with deck $\singletonDeletion{P}=\Set{P}$.

However, the Cantor set has  deck $\singletonDeletion{C}=\singleton{C \setminus \singleton{0}}$, and this coincides with the deck of $C \setminus \singleton{0}$. Hence the Cantor set is not reconstructible. Define  a property $\script{P}$ of topological spaces to be reconstructible if for all spaces $X$ and $Z$, $\singletonDeletion{X}=\singletonDeletion{Z}$ implies `$X$ has $\script{P}$ if and only if $Z$ has $\script{P}$'. Then the above example also shows that compactness is not a reconstructible property.

A \emph{continuum} is a compact, connected  Hausdorff space. The \emph{weight} of a space $X$, denoted $w(X)$, is the minimal size of a basis for the topology of $X$. Our key reconstruction result (Theorem~\ref{bigreconstructionresult}) is that a continuum $X$ is reconstructible if $w(X)<|X|$. 
Since non-trivial metrizable continua have countable weight but uncountable cardinality, we deduce: every compact Hausdorff space containing a metrizable subcontinuum with non-empty interior is reconstructible (Corollary~\ref{met_subctm}), and every compact metrizable space which is the countable union of connected subsets is reconstructible (Corollary~\ref{cptmet_ctbl_conn}). It immediately follows that  any finite sum of metrizable continua is reconstructible, including all finite topological  graphs. 

An interesting open problem is whether infinite (in other words, non-compact) topological graphs are topologically reconstructible.  Although infinite counterexamples to the (Graph) Reconstruction Conjecture are not hard to find, it is another longstanding open problem to determine whether all infinite, locally finite connected graphs are reconstructible, see Problem~1 in \cite{nash}. We can show (Corollary~\ref{graph_like}) that  compact graph-like spaces are reconstructible, and so the topological reconstruction of infinite graphs is possible provided we include information about the ends of the graph. 

It is also unclear to the authors whether the restriction in the key Theorem~\ref{bigreconstructionresult} to continua with weight strictly less than cardinality is necessary. We are able to show that in many cases it is not needed. In \cite{part2} the authors build on the work of this paper and characterise exactly when compact metrizable spaces are reconstructible.

The paper is organised as follows. Section~\ref{section2} contains the basic background results on reconstructing properties necessary for our later discussion. 
In Section~\ref{section3}, we recall elements of the theory of finite compactifications, and then prove our fundamental reconstruction result, connecting reconstructibility of compact Hausdorff spaces to the existence of a card with a maximal finite compactification. In Section~\ref{section4}, we show that every Hausdorff continuum $X$ with $w(X) < \cardinality{X}$ has a card with a maximal 1- or 2-point compactification and hence is reconstructible. The crucial technical step towards this result is an internal characterisation when removing a point from a connected compact space leaves behind a card  with a maximal finite compactification, connecting the topological reconstruction problem to classical work by Whyburn from the 1930's on special points in subsets of the plane.
The last section, Section~\ref{section5}, contains examples illustrating the previous results. We also discuss why it follows that every graph, and more generally, every compact graph-like space is reconstructible in the topological sense. For all standard topological notions see Engelking's \emph{General Topology} \cite{Eng}. To exclude trivial counter-examples, we will assume that all spaces discussed in this paper contain at least three points.

\section{Topologically reconstructible properties}
\label{section2}

This section contains background results on topological reconstruction. Most importantly, we see that local properties, so topological properties like local compactness or local connectedness, are reconstructible in $T_1$ spaces, and further, that for compact Hausdorff spaces, reconstructing compactness is equivalent to reconstructing the space itself. Additional information on the topological reconstruction problem can be found in \cite{recpaper}.

\subsection{Reconstructing separation axioms} 
Recall  that a $T_1$\emph{-space} is a topological space in which singletons are closed, and a \emph{Hausdorff space} is a topological space in which every two points can be separated by disjoint open subsets (and see \cite{Eng} for other separation axioms).

\begin{mythm}[{\cite[3.1]{recpaper}}]
\label{mythm8}
For topological spaces containing at least three points, hereditary separation axioms are reconstructible. Normality (resp.\ paracompactness) is reconstructible provided that the space has at least one normal (paracompact) card. \qed
\end{mythm}

In particular, it follows that being a $T_1$ or a Hausdorff space is reconstructible.

\subsection{Reconstructing local properties}

For a topological property $\mathcal{P}$, a topological space $X$ and a point $x \in X$, we say that $X$ is  \emph{locally} $\mathcal{P}$ \emph{at} $x$ (and write $(x,X) \vDash \mathcal{P}$) if for every open neighbourhood $U$ of $x$ there is $A \subset X$ such that $x \in \interior{A} \subset A \subset U$ and $A$ is $\mathcal{P}$. Similarly, we say $X$ is   \emph{locally} $\mathcal{P}$ if $X$ is locally $\mathcal{P}$ at $x$ for all $x \in X$. 

Examples of local properties are local compactness, local connectedness and local metrizability, but also seemingly global properties like having no isolated points or being zero-dimensional in $T_1$ spaces.

\begin{mylem}
\label{localproperty2}
In the realm of $T_1$-spaces, the number of points $x$ such that the space is (not) locally $\mathcal{P}$ at $x$ is reconstructible for all topological properties $\mathcal{P}$.
\end{mylem}
\begin{proof}
Let $X$ be a $T_1$-space. Noting that local properties are hereditary with respect to open subspaces, and a neighbourhood basis of $x$ in the card $X \setminus \Set{y}$ is a neighbourhood basis of $x$ in $X$, we have
$$(x,X) \vDash \mathcal{P} \; \textnormal{ if and only if } \; (x,Y) \vDash \mathcal{P}$$ for all cards $Y = X \setminus \singleton{y}$ with $y \neq x$. Since discrete spaces are reconstructible, we may assume without loss of generality that $X$ is infinite. It follows 
\[\cardinality{\set{x \in X}:{(x,X) \vDash \mathcal{P}}} = \max \set{ \cardinality{\set{y \in Y}:{(y,Y) \vDash \mathcal{P}}}}:{Y \in \singletonDeletion{X}}.\]
Indeed, by our initial observation, we have $\geq$ always. If $\cardinality{\set{x \in X}:{(x,X) \vDash \mathcal{P}}}$ is finite, then for any $y \in X$ with $(y,X) \not\vDash \mathcal{P}$, the card $Y=X \setminus \singleton{y}$ witnesses equality in the above equation. And if $\cardinality{\set{x \in X}:{(x,X) \vDash \mathcal{P}}}$ is infinite, then every card $Y \in \singletonDeletion{X}$ witnesses equality. 
\end{proof}

\begin{mycor}[{\cite[3.3]{recpaper}}]
\label{localproperty}
In the realm of $T_1$-spaces, ``being locally $\mathcal{P}$'' is reconstructible for all topological properties $\mathcal{P}$. \qed
\end{mycor}

\subsection{Reconstruction and compactness} 

Even though compactness is not reconstructible, compactness can be helpful for reconstructing spaces.

\begin{mythm}[{\cite[5.1]{recpaper}}]
\label{mythm20}
Every compact Hausdorff space containing isolated points is reconstructible. \qed
\end{mythm}

Indeed, this follows from Theorem \ref{mythm8} plus the fact that a card of a Hausdorff space is compact only if the deleted point was isolated, as compact subsets of Hausdorff spaces are closed.

The next result is essentially due to the uniqueness of the Alexandroff 1-point compactification.
\begin{mythm}[{\cite[5.2]{recpaper}}]
\label{compactreconstruction} \label{thm:allReconstructionsCompact}
If a compact Hausdorff space only has compact reconstructions then it is reconstructible. \qed
\end{mythm}

\section{Connecting reconstructibility and finite compactifications}

\label{section3}

This section is devoted to establishing the following sufficient condition for reconstructibility: every compact Hausdorff space that has a card with a maximal finite compactification is reconstructible (Theorem~\ref{fincompactrec}).

\subsection{Maximal finite compactifications}

A Hausdorff compactification $\gamma X$ of a space $X$ is called an $N$\emph{-point compactification} (for $N \in \N$) if its remainder $\gamma X \setminus X$ has cardinality $N$.  A \emph{finite compactification} of $X$ is an $N$-point compactification for some $N \in \N$. If a space has an $N$-point compactification, then by identifying points in the remainder we can obtain $M$-point compactifications for all $1 \leq M \leq N$. 

We say $\nu X$ is a \emph{maximal} $N$-point compactification if no other finite compactification $\gamma X$ has a strictly larger remainder, i.e.\ whenever $\cardinality{\gamma X \setminus X}=M$ then $M \leq N$. We agree on the convention that a compact Hausdorff space has a maximal $0$-point compactification. Here and in the following, we denote by $\alpha X$ the Alexandroff 1-point compactification of a locally compact Hausdorff space $X$. Note that a space $X$ has a finite compactification if and only if it is locally compact and Hausdorff.  For general information on compactifications, we suggest \cite[\S3.5]{Eng} or \cite{Chandler}.

There are many examples of topological spaces with a maximal finite compactification. For example, the real line has a maximal $2$-point compactification, and Euclidean spaces $\R^n$ for $n \geq 2$ have a maximal $1$-point compactification (see \cite[\S3]{Magill1} or \cite[6L.3]{Rings}).  Similarly, if we view a finite graph $G$ as a 1-dimensional cell-complex and delete a vertex $v$ then $G \setminus \singleton{v}$ has a maximal $N$-point compactification where $N=\operatorname{deg}(v)$.
 For different types of examples, note that the first uncountable ordinal $\w_1$ and, more generally, every ordinal $\alpha$ of uncountable cofinality are spaces with a maximal $1$-point compactification, because their Stone-\v{C}ech compactification is $\alpha +1$. A similar example is given by the Tychonoff plank \cite[3.12.20]{Eng}, which also has a one-point Stone-\v{C}ech remainder.


In order to relate one further class of examples of maximal finite compactifications, recall that the \emph{Freudenthal compactification} of a locally compact space is the unique maximal compactification with zero-dimensional remainder \cite[VI.3.7]{Aarts}. As finite remainders are discrete (and hence zero-dimensional), we obtain the following result.

\begin{mythm}
A locally compact Hausdorff space has a maximal finite compactification if and only if it has a finite Freudenthal compactification; in this case, both compactifications are equivalent. \qed
\end{mythm}

It follows that a locally finite, connected graph (viewed as a topological space) has a maximal finite compactification if and only if it has finitely many \emph{ends} (in the sense of Diestel,  \cite[\S8]{graphtheory}). Indeed, the ends of such a graph correspond naturally to points in the remainder of the Freudenthal compactification \cite[8.5.4]{graphtheory}.

Magill  \cite{MagillCtble}, found an attractive characterisation of maximal finite compactifications: a locally compact space has a maximal finite compactification if and only if it does not possess a compactification with countably infinite remainder.  
The next theorem, also due to Magill, gives an internal characterisation when a space has an $N$-point compactification for given $N \in \N$.

\begin{mythm}[{\cite[2.1]{Magill1}}]
\label{magill}
For $N \geq 1$, a Tychonoff space X has an $N$-point compactification if and only if $X$ is locally compact and contains $N$ non-empty open pairwise disjoint subsets $G_i$ ($1 \leq i \leq N$) such that $K = X \setminus \Union_{i=1}^N G_i$ is compact, but the union $K \cup G_i$ is non-compact for each $i$. \qed
\end{mythm}

Following \cite[2.2]{Magill1}, we call a family $\singleton{K} \cup \Set{G_i}_{i=1}^N$ with the above properties an \emph{$N$-star}. 
In context of the topological reconstruction problem, we are interested in the case when cards of compact Hausdorff spaces have maximal finite compactifications. For $N \geq 1$, we say a point $x \in X$ is $N$-\emph{splitting} in $X$ if $X \setminus \singleton{x} = X_1 \oplus \ldots \oplus X_N$ such that $x \in \closure{X_i}$ for all $i \leq N$. Further, we say that $x$ is \emph{locally $N$-splitting} (in $X$) if there exists a neighbourhood $U$ of $ x$, i.e.\ a set $U$ with $x \in \interior{U}$, such that $x$ is $N$-splitting in $U$. Note that whenever we have neighbourhoods $U$ and $V$ of $x$ such that $V \subset U$ and $x$ is $N$-splitting in $U$, then $x$ is $N$-splitting in $V$ as well.

\begin{mylem}
\label{lem:easytranslation}
A card $X \setminus \singleton{x}$ of a compact Hausdorff space $X$ has an $N$-point compactification if and only if $x$ is locally $N$-splitting in $X$.
\end{mylem}

\begin{proof}
As the result is clear for isolated $x$, we may suppose that $x$ is non-isolated.

For the direct implication, assume that $X \setminus \singleton{x}$ has an $N$-point compactification for some $N \geq 1$. By Theorem \ref{magill}, the space $X\setminus \singleton{x}$ contains an $N$-star $\singleton{K} \cup \Set{G_i}_{i=1}^N$. We claim that $U = X \setminus K$ is a neighbourhood as required. Indeed, since $K$ is compact and hence closed in the Hausdorff space $X$, the set $U$ is an open neighbourhood of $x$. Because $K\cup G_i$ is non-compact and closed in $X \setminus \singleton{x}$ (since it is a complement of open sets), we have $x \in \closureIn{G_i}{X}$. It follows that $U \setminus \singleton{x} = G_1 \oplus \cdots \oplus G_N$ is as desired.
 
Conversely, if there is a neighbourhood $U$ of $x$ such that $U \setminus \singleton{x} = U_1 \oplus \cdots \oplus U_N$ with $x \in \closureIn{U_i}{X}$ for all $i \leq N$, then $\Set{X \setminus U} \cup  \Set{U_i}_{i=1}^N$ is an $N$-star of $X \setminus \singleton{x}$. Indeed, no set $U_i \cup X\setminus U$ can be compact because it is not closed in $X$. Thus, $X \setminus \singleton{x}$ has an $N$-point compactification by Theorem \ref{magill}.
\end{proof}

\subsection{Cards with finite compactifications}
\label{section213}

We begin by introducing a new cardinal invariant for locally compact spaces. For a point $x$ in a locally compact Hausdorff space $X$ and a compact neighbourhood $D$ of $x$, we let $\localends{x,D}$ be the largest number $N \in \N=\{0,1,2,\ldots\}$ such that $D \setminus \singleton{x}$ has a finite $N$-point compactification. If there is no largest such number, we say $\localends{x,D}=\infty$. Thus, $\localends{x,D}$ takes values in $\N \cup \Set{\infty}$.  We observe that $\localends{x,D}=0$ if and only if $x$ is isolated.

\begin{mylem}
\label{endswelldef}
Let $X$ be a locally compact Hausdorff space and let $x \in X$. For any two compact neighbourhoods $D_1$ and $D_2$ of $x$, we have $\localends{x,D_1} = \localends{x,D_2}$.
\end{mylem}

\begin{proof}
Since the statement is symmetric, it suffices to prove only one inequality, say ``$\leq$". For this, we have to show that if $D_1 \setminus \singleton{x}$ has an $N$-point compactification, then so does  $D_2 \setminus \singleton{x}$.

So assume that $D_1 \setminus \singleton{x}$ has an $N$-point compactification. By Lemma \ref{lem:easytranslation}, there is a $D_1$-open neighbourhood $U \subset D_1$ of $x$ witnessing that $x$ is locally $N$-splitting in $D_1$. Note that we may assume that $U \subset \interiorIn{D_1}{X} \cap \interiorIn{D_2}{X}$. But then $U \subset D_2$ witnesses that $x$ is locally $N$-splitting in $D_2$. Another application of Lemma \ref{lem:easytranslation} completes the proof.
\end{proof}

Thus, the value of $\localends{x,D}$ is independent of the chosen neighbourhood $D$. In the following, we denote by $\localends{x}=\localends{x,X}$ the value of $\localends{x,D}$ for an arbitrary compact  $D \subset X$ with $x \in \interior D$.

\begin{mycor}
\label{easypiesie}
For a point $x$ in a compact Hausdorff space $X$ we have $\localends{x,X}=N$ if and only if the card $X\setminus \singleton{x}$ has a maximal $N$-point compactification. \qed
\end{mycor}

\begin{mycor}
\label{easypiesie2}
For a point $x$ in a locally compact, non-compact Hausdorff space $X$ we have $\localends{x,X}=N$ if and only if $\alpha X\setminus \singleton{x}$ has a maximal $N$-point compactification. \qed
\end{mycor}

\begin{mycor}
A locally compact, non-compact space $X$ has a maximal $N$-point compactification if and only if $\localends{\infty,\alpha X} = N \geq 1$. \qed
\end{mycor}

We now investigate under which conditions a subspace of a compact space has a maximal finite compactification when deleting two points $x$ and $y$ instead of just one. As one might expect, we have to add the individual values of $\localends{x,X}$ and $\localends{y,X}$. This will be formalised in the next two results.

 For a locally compact space $X$ we write $\Remainder{n}{X}= \set{x \in X}:{\localends{x,X}=n}$, where $n \in \N \cup \singleton{\infty}$. We make the convention that $n + \infty = \infty = \infty +n$ for all $n \in \N \cup \singleton{\infty}$. Following \cite[2.4.12]{Eng}, for $A \subset X$ we denote by $X /A$ the quotient space that identifies $A$ with a single point. It is easy to verify that if $X$ is locally compact Hausdorff and $A \subset X$ is compact, then $X /A$ is locally compact Hausdorff as well.

\begin{mylem}
\label{newname}
Let $X$ be a locally compact Hausdorff space $X$. If $x \in \Remainder{n}{X}$ and $y \in \Remainder{m}{X}$ then $\singleton{x,y} \in \Remainder{n+m}{X/\Set{x,y}}$ for all $n,m \in \N \cup \singleton{\infty}$.
\end{mylem}

\begin{proof}
Suppose first that $n$ and $m$ are finite. We may assume that $X$ is compact, as otherwise we can use Corollary \ref{easypiesie2} and work in $\alpha X$. Since the quotient $X/\Set{x,y}$ is compact Hausdorff, it suffices by Corollary \ref{easypiesie} to show that $X \setminus \Set{x,y}$ has a maximal $(n+m)$-point compactification. 

Fix disjoint compact sets $C_x$ and $C_y$ such that $x \in \interior{C_x}$ and $y \in \interior{C_y}$. Since by assumption $C_x \setminus \singleton{x}$ has an $n$-point compactification,  Lemma \ref{lem:easytranslation} implies that there is an $X$-open set $U$ with $x \in U \subset \interior{C_x}$ witnessing that $x$ is locally $n$-splitting in $C_x$. Similarly, there is an open set $V$ with $y \in V \subset \interior{C_y}$ witnessing that $y$ is locally $m$-splitting in $C_y$. It follows that $(U \cup V)/\Set{x,y}$ is a neighbourhood witnessing that the collapsed point $\Set{x,y}$ is locally $(n+m)$-splitting in $ X / \Set{x,y}$. By Lemma \ref{lem:easytranslation}, the space $X \setminus \Set{x,y}$ has an $(n+m)$-point compactification. 

To see that this compactification is maximal, assume for a contradiction that $X \setminus \Set{x,y}$ has an $N$-point compactification for $N = n+m+1$. By  Lemma \ref{lem:easytranslation}, there is an open neighbourhood $W$ such that $\Set{x,y} \in W \subset X/\Set{x,y}$ witnessing that $\Set{x,y}$ is locally $N$-splitting in $X/\Set{x,y}$, i.e.\
$$ W \setminus \Set{x,y} = W_1 \oplus \ldots \oplus W_N \; \textnormal{ such that } \Set{x,y} \cap \closureIn{W_i}{X} \neq \emptyset.$$
Now let $I_x = \set{1 \leq i \leq N}:{x \in \closure{W_i}}$ and $I_y=\set{1 \leq i \leq N}:{y \in \closure{ W_i}}.$
 By the pigeon hole principle, we have either $\cardinality{I_x} > n$ or $\cardinality{I_y} > m$. Hence, the open sets
 $$U= \singleton{x} \cup \bigcup_{i \in I_x} W_i \; \textnormal{ and } \;  V= \singleton{y} \cup \bigcup_{i \in I_y} W_i$$ 
witness either that $x$ is locally $(n+1)$-splitting or that $y$ is locally $(m+1)$-splitting in $X$, contradicting that $x \in \Remainder{n}{X}$ and $y \in \Remainder{m}{X}$. 

This completes the proof for finite $n$ and $m$. In the case where either $n$ or $m$ is infinite, simply observe that the first part of the proof implies that $X \setminus \Set{x,y}$ has arbitrarily large finite compactifications.
\end{proof}

\begin{mycor}
\label{thisisobvious}
Let $X$ be a compact Hausdorff space and $x \neq y \in X$. Then $X \setminus \Set{x,y}$ has a maximal $N$-point compactification (for $N \in \N \cup \singleton{\infty}$) if and only if $\localends{x,X} + \localends{y,X} = N$. \qed
\end{mycor}

\subsection{The reconstruction result} Equipped with the cardinal invariant $\localends{x}$, we now prove our Reconstruction Result for Compact Spaces. As $\localends{x}$ has been defined in terms of local neighbourhoods, it is reconstructible by Lemma~\ref{localproperty2}.

\begin{mylem}
\label{schwupssdiwupps}
In a locally compact Hausdorff space $X$, $\cardinality{\Remainder{n}{X}}$ is reconstructible for all $n \in \N \cup \Set{\infty}$. \qed
\end{mylem}

\begin{mythm}[Reconstruction Result for Compact Spaces]
\label{fincompactrec}
Every compact Hausdorff space that has a card with a maximal finite compactification is reconstructible.
\end{mythm}

\begin{proof}
Suppose $X$ is a compact Hausdorff space that has a card with a maximal finite compactification. Using Corollary \ref{easypiesie}, we see that $\Remainder{n}{X} \neq \emptyset$ for some $n \in \N$. Let $N$ be minimal such that $\Remainder{N}{X} \neq \emptyset$ and fix $x \in \Remainder{N}{X}$.

Suppose for a contradiction that $Z$ is a non-compact reconstruction of $X$. By Theorem \ref{localproperty}, the space $Z$ is locally compact Hausdorff. Find $z \in Z$ such that $Z \setminus \singleton{z} \cong X \setminus \singleton{x}$. Then $Z \setminus \singleton{z}$ has a maximal $N$-point compactification, so Corollary \ref{thisisobvious} implies that $\localends{z,Z} + \localends{\infty,\alpha Z} = N$. Since $Z$ is non-compact, we have $\localends{\infty,\alpha Z} \geq 1$ and hence $\localends{z,Z} < N$. Thus, $\Remainder{n}{Z} \neq \emptyset$ for some $n < N$, and hence $\Remainder{n}{X} \neq \emptyset$ for the same $n < N$ by Lemma \ref{schwupssdiwupps}, contradicting the minimality assumption on $N$. 

Thus, any reconstruction of $X$ is compact. It follows from Theorem \ref{compactreconstruction} that $X$ is reconstructible.
\end{proof}

\section{Many continua are reconstructible}
\label{section4}

We show that every continuum $X$ such that $w(X) <|X|$ has a card with a maximal $1$- or $2$-point compactification. Applying the Reconstruction Result for Compact Spaces, it follows that every such space is  reconstructible. 

Our plan goes as follows. We first strengthen Lemma~\ref{lem:easytranslation} to show that for a continuum $X$, a card $X \setminus \singleton{x}$ has an $N$-point compactification if and only if $x$ is locally $N$-separating---a condition which is (formally) substantially stronger than locally $N$-splitting. 

We then use the given restriction on weight {\sl versus} cardinality and a counting argument to show that only `few' points can be locally $N$-splitting for $N \geq 3$. So in fact `most' cards have a maximal $1$- or $2$-point compactification.

\subsection{Background from continuum theory}
Our basic reference for continuum theory is \cite{Cont}.
In a topological space $X$, the (connected) \emph{component} of a point $x$, denoted by $C(x)=C_X(x)$, is the union over all connected subspaces of $X$ containing $x$. 
The \emph{quasi-component} of a point $x$, denoted by $Q(x)$, is the intersection over all clopen sets of $X$ containing $x$. Components and quasi-components are closed and $C(x) \subset Q(x)$ always holds. We will need the following classical results from continuum theory.

\begin{mylem}[\v{S}ura-Bura Lemma {\cite[6.1.23]{Eng}}]
\label{SuraBura}
In a compact Hausdorff space, components and quasi-components agree, i.e.\ $C(x) = Q(x)$ for all $x \in X$.\qed
\end{mylem}

\begin{mylem}[Second \v{S}ura-Bura Lemma]
\label{adaptsurabura}
Let $X$ be a compact Hausdorff space, and $C \subset X$ a component. Then $C$ has a clopen neighbourhood base in $X$, i.e. for every open set $U \supset C$ there is a clopen set $V$ such that $C \subset V \subset U$.
\end{mylem}

\begin{proof}
The proof is a straightforward application of the \v{S}ura-Bura Lemma \ref{SuraBura} and compactness. For details, see e.g.\ \cite[A.10.1]{Mill01}.
\end{proof}

\begin{mylem}[Boundary Bumping Lemma {\cite[6.1.25]{Eng}}]
\label{boundarybumping}
The closure of every component of a non-empty proper open subset $U$ of a Hausdorff continuum intersects the boundary of $U$, i.e.\ $\closure{C_U(x)} \setminus U \neq \emptyset$ for all $x \in U$. \qed
\end{mylem}

\subsection{Characterising maximal finite compactifications} 
Let $X$ be a topological space. A point $x \in X$ is called \emph{separating} in $X$ if $X \setminus \singleton{x}$ has a disconnection $ A_1 \oplus A_2$ such that both $A_1$ and $A_2$ intersect $C_X(x)$. A point $x$ is called \emph{locally separating} in $X$ if there is a neighbourhood $U$ of $x$ such that $x$ is separating  in $U$. These definitions are due to Whyburn \cite[III.8-9]{why}. 

We extend these definitions by saying that $x$ is \emph{$N$-separating} in $X$ if $X \setminus \singleton{x}$ has a disconnection into $N$ (clopen) sets $ A_1 \oplus \cdots \oplus A_N$ such that all $A_i$ intersect $C_X(x)$. Similarly, we say $x$ is \emph{locally $N$-separating} in $X$ if there is a neighbourhood $U$ of $x$ such that $x$ is $N$-separating in $U$. 

A point $x$ of a connected space $X$ is called \emph{cut point} if $X \setminus \singleton{x}$ is disconnected, and it is called \emph{$N$-cut point} if $X \setminus \singleton{x}$ splits into at least $N$ non-empty disjoint clopen sets. Note that in a connected space $X$, the notion of $N$-cut point and $N$-separating point coincide. A point of some topological space is called \emph{sub cut point} if it is a cut point of some connected subspace \cite{DowHart}. More generally, we call a point a \emph{sub N-cut point} if it is an $N$-cut point of some connected subspace.

\begin{figure}[h]
\begin{tikzpicture}[scale=1.5]
\draw [thick] (0,0) -- (0,2);
\draw [thick] (0,0) -- (0,-1.5);
\draw [thick] (0,0) -- (2,0);
\draw [thick] (0,0) -- (-1.5,0);

\foreach \x in {0,1,2,3,...,5}{
\pgfmathsetmacro\result{.7*(0.8^\x)}
 \draw[thick] (\result,\result) -- (\result,2);
  \draw[thick] (\result-.014,\result) -- (2,\result);
}

\node[rectangle, black, fill=black, draw,inner sep=0, minimum height=.2pt,minimum width=.2pt] (Lim1) at (.05,1) {}; 
\node[rectangle, black, fill=black, draw, inner sep=0,minimum height=.2pt,minimum width=.2pt] (Lim2) at (.15,1) {}; 
\node[rectangle, black, fill=black, draw, inner sep=0, minimum height=.2pt,minimum width=.2pt] (Lim3) at (.1,1) {}; 

\node[rectangle, black, fill=black, draw, inner sep=0,minimum height=.2pt,minimum width=.2pt] (Lim4) at (1,.05) {}; 
\node[rectangle, black, fill=black, draw, inner sep=0,minimum height=.2pt,minimum width=.2pt] (Lim5) at (1,.15) {}; 
\node[rectangle, black, fill=black, draw, inner sep=0, minimum height=.2pt,minimum width=.2pt] (Lim6) at (1,.1) {};

\draw[thick] (-1.5,0) arc (180:270:1.5);

\node[circle, black, fill=white, draw, inner sep=1pt,minimum width=1pt] (x) at (0,0) {}; 
\node (x2) at (.2,-.2) {$x$}; 
\end{tikzpicture}
\captionsetup{singlelinecheck=off}
\caption[blablabla]{A (disconnected) subset $X$ of the plane and a designated point $x$ which is
\begin{itemize}
\itemsep0em 
\item $2$-separating but not $3$-separating,
\item locally $3$-separating but not locally $4$-separating, and
\item a sub $4$-cut point, but not a sub $5$-cut point.
\end{itemize}
}
\end{figure}

Next, we show that ``locally $N$-separating'' is a stronger condition than ``locally $N$-splitting''. We then show as a consequence of the Boundary Bumping Lemma that in Hausdorff continua, both notions coincide.

\begin{mylem}
Let $x$ be a point in a $T_1$ space. If $U$ is an open set witnessing that $x$ is locally $N$-separating, i.e. $U\setminus \singleton{x}=U_1 \oplus \cdots \oplus U_N$ and $U_i \cap C_U(x) \neq \emptyset$ for all $i$, then $x$ is contained in the boundary of $U_i \cap C_U(x)$ for all $i$.
\end{mylem}

\begin{proof}
Since $X$ is $T_1$, every $U_i$ is $X$-open and every $U_i \cup \singleton{x}$ is $U$-closed. If $x$ does not lie in the closure of $C_i=U_i \cap C_U(x)$ for some $i$, this $C_i$ would be a non-trivial clopen subset of the connected set $C_U(x)$, a contradiction. Indeed, $U_i$ witnesses that $C_i \subset C_U(x)$ is open, and if $x \notin \closure{C}_i$, then $C_i$ is closed in $ \p{U_i \cup \singleton{x}} \cap C_U(x)$, which in turn in closed in $C_U(x)$.
\end{proof}

\begin{mycor}
\label{cor:Nsep}
In $T_1$ spaces, locally $N$-separating points are locally $N$-splitting. \qed
\end{mycor}

\begin{mycor}
\label{cor:Nsep2}
Suppose $V \subset U$ are neighbourhoods of a point $x$ in a $T_1$ space. If $x$ is $N$-separating in $U$ then $x$ is $N$-separating in $V$. \qed
\end{mycor}

\begin{mylem}
\label{fintranslationhelp}
A point of a Hausdorff continuum is locally $N$-separating if and only if it is locally $N$-splitting.
\end{mylem}

\begin{proof}
The direct implication follows from Corollary \ref{cor:Nsep}. For the converse, let $X$ be a Hausdorff continuum and suppose $x$ is locally $N$-splitting in $X$. Find a witnessing compact neighbourhood $U$ of $x$, i.e.\ $U \setminus \singleton{x} = U_1 \oplus \cdots \oplus U_N$ and $x \in \closure{U}_i$ for all $i$.

For every $i$, consider the compact space $A_i = U_i \cup \singleton{x}$. To prove that $x$ is locally $N$-separating in $X$, it suffices to show that $C_{A_i}(x)$, the component of $x$ in $A_i$, is non-trivial. Suppose it is trivial, then by the Second \v{S}ura-Bura Lemma \ref{adaptsurabura}, there is an $A_i$-clopen neighbourhood $V$ of $x$  such that $x \in V \subset \interior{U} \cap A_i$. Observe that $V \setminus \singleton{x}$ is a non-empty open subset of $X$ containing only the point $x$ in its boundary. Hence, by the Boundary Bumping Lemma \ref{boundarybumping}, every component of $V \setminus \singleton{x}$ limits onto the boundary point $x$, so that $C_{A_i}(x)$ must be non-trivial.
\end{proof}

\begin{mylem}
\label{fintranslation}
A card $X \setminus \singleton{x}$ of a Hausdorff continuum $X$ has an $N$-point compactification if and only if $x$ is locally $N$-separating in $X$.
\end{mylem}
\begin{proof}
By Lemmas \ref{lem:easytranslation} and \ref{fintranslationhelp}.
\end{proof}

\subsection{Reconstruction of  continua}

The \emph{density} of a space, denoted $d(X)$, is the minimum size of a dense subset of $X$. Note that the density of any space is no more than its weight.

\begin{mylem}[For metrizable continua: Whyburn {\cite{Whyburn}}]
\label{lemma01}
The number of $3$-cut points in a connected $T_1$-space $X$ is not larger than the density of $X$. 
\end{mylem}

\begin{proof}
See \cite[7.3]{recpaper}.
\end{proof}

\begin{mylem}[For separable metric spaces: Whyburn {\cite[III(8.1)]{why}}]
\label{lemma02}
In a $T_1$ space $X$, the number of components containing separating points of $X$ does not exceed the weight of $X$. \end{mylem}

\begin{proof}
Suppose for a contradiction that in a $T_1$ space $X$ of weight $\kappa$ there is a collection of distinct components $\set{C_\alpha}:{\alpha < \kappa^+}$, each containing a separating point $x_\alpha$ of $X$. For all $\alpha$, fix a disconnection $X \setminus \singleton{x_\alpha} = U_\alpha \oplus V_\alpha$ and points $a_\alpha \in U_\alpha \cap C_\alpha$ and $b_\alpha \in V_\alpha \cap C_\alpha$, witnessing that $x_\alpha$ is separating in $X$.

Now consider the space $Y=\set{(a_\alpha,b_\alpha)}:{\alpha < \kappa^+} \subset X^2$. Since $X^2$ has weight $\kappa$, the space $Y$ cannot be discrete. Thus, there exists $(a_\gamma,b_\gamma)$ which is contained in the closure of $Y \setminus \singleton{(a_\gamma,b_\gamma)}$. Considering the card $X \setminus \singleton{x_\gamma} = U_{\gamma} \oplus V_{\gamma}$, we note that for all $\alpha \neq \gamma$ the connected component $C_\alpha$ is completely contained in one of $U_{\gamma}$ or $V_{\gamma}$. Thus, the sets
$$ Y_U=\set{(a_\alpha,b_\alpha)}:{C_\alpha \subset U_{\gamma}} \quad \text{and} \quad Y_V=\set{(a_\alpha,b_\alpha)}:{C_\alpha \subset V_{\gamma}}$$
split $Y \setminus \singleton{(a_\gamma,b_\gamma)}$ into two disjoint sets, and hence $(a_\gamma,b_\gamma)$ will be in the closure of one of them, say, $Y_U$. However, all $b_\alpha$ with $(a_\alpha,b_\alpha) \in Y_U$ lie in $U_{\gamma}$, yielding 
$b_\gamma \in \closure{\set{b_\alpha}:{(a_\alpha,b_\alpha) \in Y_U}} \subset \closure{U}_\gamma= U_{\gamma} \cup \singleton{x_\gamma},$
a contradiction. 
\end{proof}

\begin{mylem}[For locally compact separable metric spaces: Whyburn {\cite[III(9.2)]{why}}]
\label{bigreconstructionresult2}
The number of locally $3$-separating points in a $T_1$ space $X$ does not exceed the weight of $X$.
\end{mylem}

\begin{proof}
Fix a base $\script{B}$ of size $w(X)$ and suppose towards a contradiction that there is a collection $S \subset X$ with $\cardinality{S} > w(X)$ such that every point $x \in S$ is locally $3$-separating in $X$. By Corollary~\ref{cor:Nsep2}, for every $x \in S$ there is a basic open neighbourhood $U_x \in \script{B}$ which it $3$-separates. Since $\cardinality{S} > w(X)$, there exists a basic open $U \in \script{B}$ and a subset $S' \subset S$ with $\cardinality{S'} > w(X)$ such that $U=U_x$ for all $x \in S'$.

Applying Lemma~\ref{lemma02} to $U$, we see that there is a component $C$ of $U$ such that $\cardinality{S' \cap C} > w(X)$. But since $d(C)\leq w(X)$, the connected $T_1$ space $C$ now contains more than $d(C)$ many $3$-separating points. Since $C$ is connected, all these $3$-separating points are in fact $3$-cut points, contradicting Lemma~\ref{lemma01}. 
\end{proof}

\begin{mythm}
\label{bigreconstructionresult}
Every Hausdorff continuum $X$ with $w(X) < \cardinality{X}$ is reconstructible.
\end{mythm}

\begin{proof}
By Lemma~\ref{bigreconstructionresult2} and Lemma~\ref{fintranslation}, every Hausdorff continuum $X$ with $w(X) < \cardinality{X}$ has cards with maximal $1$- or $2$-point compactifications. Thus, they are reconstructible by Theorem~\ref{fincompactrec}.
\end{proof}

\begin{mycor}\label{met_subctm}
Every compact Hausdorff space containing a metrizable subcontinuum with non-empty interior is reconstructible.
\end{mycor}
\begin{proof}
Let $X$ be a compact Hausdorff space and $D$ a metrizable subcontinuum with non-empty interior. If $\interior{D}$ contains isolated points, then $X$ is reconstructible by Theorem~\ref{mythm20}. Otherwise, $\interior{D}$ is completely metrizable without isolated points \cite[4.3.23]{Eng}, so has cardinality $\cont$ \cite[4.5.5]{Eng}. Since $D$ has countable weight, by Lemma~\ref{bigreconstructionresult2}, there is $x \in \interior{D}$ such that $D \setminus \singleton{x}$ does not have a 3-point compactification. Lemma~\ref{endswelldef} implies that $X \setminus \singleton{x}$ does not have a $3$-point compactification. Thus, $X$ is reconstructible by Theorem~\ref{fincompactrec}
 \end{proof}

\begin{mycor}\label{cptmet_ctbl_conn}
Every compact metrizable space with at most countably many components is reconstructible.
\end{mycor}
\begin{proof}
Suppose $X$ is compact, metrizable and $X=\bigcup_n C_n$ where $C_n$ are (closed) connected components. By the Baire Category theorem some $C_n$ has non-empty interior,  so we can apply Corollary~\ref{met_subctm}.
\end{proof}

\section{Examples and Questions} 
\label{section5}

\subsection{Graphs and graph-like spaces}
That finite (in other words, compact) topological graphs are reconstructible is immediate from either of the two Corollaries (\ref{met_subctm} and \ref{cptmet_ctbl_conn}) above. 

A \emph{graph-like space} is a metrizable space $X$ with a $0$-dimensional subspace $V$ (the vertices) of $X$ so that every component of $X\setminus V$ (the edges) is an open subset of $X$ homeomorphic to the open unit interval and whose boundary in $X$ consists of only one or two points (see e.g.\ \cite{graphlike}). 
Any finite graph viewed as a 1-dimensional cell-complex is a compact graph-like space, and so is the 1-point or the Freudenthal compactification of an infinite, locally finite graph.
\begin{mycor}\label{graph_like}
Every compact graph-like space $X$ with at least one edge is reconstructible.
\end{mycor}
To see this, consider a point $x$ lying on one of the edges. Since $x$ has a compact neighbourhood homeomorphic to $[0,1]$, Corollary~\ref{met_subctm} applies.

However, without compactness the techniques for topological reconstruction presented in this paper fail, and an interesting open problem presents itself.
\begin{myquest}
\label{allgraphsreconstructible?}
Is every (locally finite) topological graph topologically reconstructible?
\end{myquest}

\subsection{Non-metrizable continua}

First let us note that numerous natural examples of non-metrizable continua are reconstructible as a consequence of Theorem~\ref{bigreconstructionresult}.
For example the Stone-\v{C}ech compactifications, $\beta \R$ and $\beta \mathbb{H}$, of the reals $\R$ and half-line $\mathbb{H}=[0,1)$, are reconstructible by Corollary~\ref{met_subctm}. But also the Stone-\v{C}ech remainder  $\Hstar = \beta \mathbb{H} \setminus \mathbb{H}$ of $\mathbb{H}$---which contains no non-trivial metrizable subcontinua--- is reconstructible: Since $w(\Hstar)= \cont  < 2^\cont = \cardinality{\Hstar}$, Theorem~\ref{bigreconstructionresult} applies. Similarly, $\Rstar = \beta \R \setminus \R$, is reconstructible.

Recall that always $w(X) \leq \cardinality{X}$ for compact Hausdorff spaces $X$ \cite[3.1.21]{Eng}. Thus, in light of Theorem \ref{bigreconstructionresult}, it remains to investigate whether Hausdorff continua  are reconstructible when weight equals cardinality. 

\begin{myquest}
\label{nonmetrHausdorffrecosntructible?}
Is every Hausdorff continuum reconstructible? Even stronger, 
does every Hausdorff continuum $X$ with $w(X) =\cardinality{X}$ have a card with a maximal finite compactification?
\end{myquest}

We now consider some examples of Hausdorff continua $X$ with $w(X)=\cardinality{X}$, and show that in each case they are reconstructible.

\subsubsection*{Sub cut-points and indecomposable continua}
 Recall that a point of a Hausdorff continuum is called \emph{sub N-cut point} if it is an $N$-cut point of some subcontinuum. The following observation is immediate from the definitions.
\begin{mylem}
\label{subcutpoints}
Let $X$ be a Hausdorff continuum. If $x$ is locally $N$-separating in $X$ then $x$ is a sub $N$-cut point of $X$. \qed
\end{mylem}

Recall that a continuum is said to be \emph{indecomposable} if it cannot be written as the union of two proper subcontinua and that it is \emph{hereditarily indecomposable} if every subcontinuum is indecomposable. A well-known example of a hereditarily indecomposable continuum is given by the \emph{pseudoarc} \cite{pseudoarc}.

\begin{mycor}
Every card of a hereditarily indecomposable Hausdorff continuum has a maximal $1$-point compactification. Hence every hereditarily indecomposable Hausdorff continuum is reconstructible.
\end{mycor}
\begin{proof}
Observe that an indecomposable continuum cannot have cut points. Thus, a hereditarily indecomposable continuum cannot contain sub cut points, and hence does not contain locally separating points by Lemma~\ref{subcutpoints}. 
\end{proof}

\subsubsection*{Linearly ordered spaces} Every linearly ordered set carries a natural topology, see e.g.\ \cite[1.7.4]{Eng}. Every point in a compact connected linearly ordered space is at most a sub $2$-cut point. Thus, compact connected linearly ordered spaces are reconstructible.
Examples with $w(X)=\cardinality{X}$ are given by the \emph{long line} \cite[3.12.19]{Eng} on $\cont + 1$, by the lexicographically ordered square or (consistently) by compact connected Souslin lines. 

\subsubsection*{Cones over continua}
Gary Gruenhage pointed out to us a further class of continua with $w(X)=\cardinality{X}$. Starting with a compact Hausdorff space $X$, the \emph{cone over} $X$ is the quotient $cone(X) = (X \times I) / (X \times \singleton{1})$ for $I=[0,1]$. If $X$ has cardinality and weight $\kappa \geq \cont$, then $cone(X)$ is a Hausdorff continuum where weight equals cardinality. However, using Lemma~\ref{subcutpoints}, any point of the form $(x,\frac12)$ is at most locally 2-separating, and hence all cones over compact Hausdorff spaces are reconstructible.

\subsubsection*{Cartesian products}
Another method of building continua with cardinality equalling weight is to use Cartesian products. We show these are always reconstructible.
\begin{mylem}
\label{noproductcuts}
Let $\cardinality{S} \geq 2$ and for all $s \in S$ suppose that $X_s$ is a non-trivial connected space. Then $X=\prod_{s \in S} X_s$ has no cut points.
\end{mylem}
\begin{proof}
The case of $\cardinality{S}=2$ is straightforward, and by induction we get the result for all finite products. So suppose that $S$ is infinite and pick a point $x$ of $X$. We claim that $X \setminus \singleton{x}$ is connected. Pick $y$ in $X$ such that $x_s \neq y_s$ for all coordinates $s$. Recall from the proof that connectedness is productive \cite[6.1.15]{Eng} that the $\sigma$-product at $y$, $\sigma(y,X) = \set{z \in X}:{\cardinality{\set{s}:{z_s\neq y_s}} \text{ is finite}}$, 
is connected and dense in $X$. Since $S$ was infinite, $x$ is not contained in the $\sigma$-product at $y$ and hence it follows that $X \setminus \singleton{x}$ has a dense connected subset, so is itself connected.
\end{proof}

\begin{mythm}
\label{noproductcuts2}
Let $\cardinality{S} \geq 2$ and for all $s \in S$ suppose that $X_s$ is a non-trivial Hausdorff continuum. Then $\prod_{s \in S} X_s$ contains no locally separating points.
\end{mythm}

\begin{proof}
Let $x \in X$ and suppose for a contradiction there is a basic open neighbourhood $U = \bigcap_{s \in F} \pi^{-1}_s(U_s)$ of $x$ (for $F \subset S$ finite) such that $x$ is a cut-point of $C_{\closure{U}}(x)$. Recall that $\closure{U}=\bigcap_{s \in F} \pi^{-1}_s( \closure{U_s})$ by \cite[2.3.3]{Eng}. 

\begin{myclmn}
We have $C_{\closure{U}}(x) = \bigcap_{s \in F} \pi^{-1}_s( C_{\closure{U_s}}(x_s))$.
\end{myclmn}

First, since connectedness is productive \cite[6.1.15]{Eng}, the set $\bigcap_{s \in F} \pi^{-1}_s( C_{\closure{U_s}}(x_s))$ is a connected subset of $\closure{U}$ containing the point $x$. Conversely, for every point $y \in \closure{U} \setminus \bigcap_{s \in F} \pi^{-1}_s( C_{\closure{U_s}}(x_s))$ we have $y_s \notin C_{\closure{U_s}}(x_s)$ for some index $s$, and therefore by the \v{S}ura-Bura Lemma \ref{SuraBura} there is a clopen subset $V$ of $\closure{U_i}$ separating $y_s$ from $C_{\closure{U_s}}(x_s)$. It follows that $\pi^{-1}_s(V) \cap \closure{U}$ is a clopen subset of $U$ separating $y$ from $x$. Thus, $y \notin C_{\closure{U}}(x)$, completing the proof of the claim.

So now, since $x$ is a cut-point of its component in $\closure{U}$, the claim implies that $x$ is a cut-point of the space $\bigcap_{s \in F} \pi^{-1}_s( C_{\closure{U_s}}(x_s))$. However, since every $C_{\closure{U_s}}(x_s)$ is non-trivial by the Boundary Bumping Lemma \ref{boundarybumping}, the component of $x$ in $\closure{U}$ is a non-trivial product of connected spaces, and therefore does not have cut-points by Lemma \ref{noproductcuts}, a contradiction.
\end{proof}

\begin{mycor}
Every non-trivial product of non-trivial Hausdorff continua is reconstructible. \qed
\end{mycor}

	  
\end{document}